\documentclass[draft]{amsart}
\usepackage{amsmath}
\usepackage{amssymb}
\usepackage{graphicx}
\newtheorem{theorem}{Theorem}[section]
\newtheorem{lemma}[theorem]{Lemma}

\newtheorem{corollary}[theorem]{Corollary}
\theoremstyle{definition}

\newtheorem{remark}[theorem]{Remark}

\numberwithin{equation}{section}

\begin{document}

\title[The quantization error for Ahlfors-David measures]
{Asymptotic local uniformity of the quantization error for Ahlfors-David probability measures}
\thanks{The author is supported by National Natural Science Foundation of China No. 11571144.}
\author{Sanguo Zhu}
\address{School of Mathematics and Physics, Jiangsu University
of Technology\\ Changzhou 213001, China.}
\email{sgzhu@jsut.edu.cn}

\subjclass[2000]{Primary 28A80, 28A78; Secondary 94A15}
\keywords{Quantization error, Ahlfors-David measure, Optimal sets, Voronoi partition, Local uniformity.}

\begin{abstract}
Let $\mu$ be an Ahlfors-David probability measure on $\mathbb{R}^q$, namely, there exist some constants $s_0>0$ and $\epsilon_0,C_1,C_2>0$ such that
$C_1\epsilon^{s_0}\leq\mu(B(x,\epsilon))\leq C_2\epsilon^{s_0}$ for all $\epsilon\in(0,\epsilon_0)$ and $x\in{\rm supp}(\mu)$.
For $n\geq 1$, let $\alpha_n$ be an $n$-optimal set for $\mu$ of order $r$ and $(P_a(\alpha_n))_{a\in\alpha_n}$ an arbitrary Voronoi partition with respect to $\alpha_n$. The $n$th quantization error $e_{n,r}(\mu)$ for $\mu$ of order $r$ is given by $e^r_{n,r}(\mu):=\int d(x,\alpha_n)^rd\mu(x)$. Write
$I_a(\alpha,\mu):=\int_{P_a(\alpha_n)}d(x,\alpha_n)^rd\mu(x),\;a\in\alpha_n$.
We prove that,  all the following three quantities
\[
\underline{J}(\alpha_n,\mu):=\min_{a\in\alpha_n}I_a(\alpha,\mu),\; \overline{J}(\alpha_n,\mu):=\max_{a\in\alpha_n}I_a(\alpha,\mu),\;e^r_{n,r}(\mu)-e^r_{n+1,r}(\mu)
 \]
 are of the same order as $\frac{1}{n}e^r_{n,r}(\mu)$. Thus, for Ahlfors-David probability measure on $\mathbb{R}^q$, our result shows that a weaker version of Gersho's conjecture holds.
\end{abstract}

\maketitle

\section{Introduction}
The quantization problem for probability measures has been studied intensively in the past years (cf. \cite{GL:02,GL:04,GL:05,GL:08,GL:12,PK:01}). One of the main aims of this problem is to study the error in the approximation of a given probability measure with discrete
probability measures of finite support, in terms of
$L_r$-metrics. We refer to \cite{GL:00} for rigorous mathematical foundations of quantization theory and \cite{GPM:02,PG:97} for promising applications of this theory. One may see \cite{BW:82,GN:98,Za:63} for the deep background of the quantization problem in information
theory and engineering technology.

Let $|\cdot|$ be a norm on $\mathbb{R}^q$ and $d$ the metric on $\mathbb{R}^q$ induced by this norm.
For each $n\in\mathbb{N}$, we write $\mathcal{D}_{n}:=\{\alpha\subset\mathbb{R}^{q}:1\leq{\rm card}(\alpha)\leq n\}$.
For a Borel probability measure $\nu$ on $\mathbb{R}^{q}$, the
$n$th quantization error for $\nu$ of order $r\in(0,\infty)$ is defined
by
\begin{eqnarray}\label{quanerror}
e_{n,r}(\nu):=\bigg(\inf_{\alpha\in\mathcal{D}_{n}}\int d(x,\alpha)^{r}d\nu(x)\bigg)^{1/r}.
\end{eqnarray}
By \cite{GL:00}, $e_{n,r}(\nu)$ equals the minimum
error in the approximation of $\nu$ by discrete probability measures
supported on at most $n$ points, in the sense of the $L_r$-metrics. If the infimum in (\ref{quanerror}) is attained at some
$\alpha\in\mathcal{D}_n$,  we call $\alpha$ an $n$-optimal set for $\nu$ of order $r$. The set of
all $n$-optimal sets for $\nu$ of order $r$ is denoted by
$C_{n,r}(\nu)$. By Theorem 4.12 of \cite{GL:00}, $C_{n,r}(\nu)$ is non-empty provided that $\int |x|^rd\nu(x)<\infty$.

Let $\alpha\subset\mathbb{R}^q$ be a finite set. A Voronoi
partition with respect to $\alpha$ means a
partition $\{P_a(\alpha):a\in\alpha\}$ of $\mathbb{R}^q$
satisfying
\begin{eqnarray*}
\{x\in\mathbb{R}^q:\;{\rm d}(x,a)<{\rm
d}\big(x,\alpha\setminus\{a\}\big)\big\}\subset P_a(\alpha)\subset
\big\{x\in\mathbb{R}^q:\;{\rm d}(x,a)={\rm d}(x,\alpha)\big\}.
\end{eqnarray*}
For the above $\alpha\subset\mathbb{R}^q$ and a Borel probability measure $\nu$, we write
\begin{equation}\label{integral}
I(\alpha,\nu):=\int d(x,\alpha)^rd\nu(x),\;\;I_a(\alpha,\nu):=\int_{P_a(\alpha)}d(x,\alpha)^rd\nu(x),\;a\in\alpha.
\end{equation}

In 1979, Gersho conjectured \cite{Ger:79} that, for an absolutely continuous probability measure $\nu$, elements of a Voronoi partition with respect to an $n$-optimal set asymptotically make equal contributions to the quantization error, namely,
\[
I_a(\alpha_n,\nu)\thicksim\frac{1}{n}e_{n,r}^r(\nu), \;a\in\alpha_n\in C_{n,r}(\nu).
\]
where, for two sequences $(a_n)_{n=1}^\infty,(b_n)_{n=1}^\infty$ of positive numbers, $a_n\thicksim b_n$ means that $a_n/b_n\to 1\;(n\to\infty)$. One may see \cite{GL:12} for some heuristic interpretations. This conjecture has been proved only for a certain class of one-dimensional absolutely continuous probability measures (cf. \cite{FP:2002}). In \cite{GL:12}, Graf, Luschgy and Pag\`{e}s proved the following weaker version of Gersho's conjecture for a large class of absolutely continuous probability measures $P$ (including some with unbounded support):
\begin{equation}\label{weaker}
I_a(\alpha_n,P)\asymp\frac{1}{n}e_{n,r}^r(P), \;a\in\alpha_n\in C_{n,r}(P),
\end{equation}
where $a_n\asymp b_n$ means that there exists a constant $C$ which is independent of $n$, such that $Cb_n\leq a_n\leq C^{-1}b_n$ for all $n\geq 1$.

Gersho's conjecture reflects a kind of asymptotic local uniformity of the quantization error with respect to $n$-optimal sets. It is certainly also significant for singular probability measures. In \cite{Zhu:09}, the author showed (\ref{weaker}) for self-similar measures $\mu$ on $\mathbb{R}^q$ with the assumption of the strong separation condition (SSC). Recall that the self-similar set associated with a family $(f_i)_{i=1}^M$ of contractive similitudes
on $\mathbb{R}^q$ refers to the unique non-empty compact set $E$ satisfying $E=\bigcup_{i=1}^Mf_i(E)$ and the self-similar measure associated with $(f_i)_{i=1}^M$ and a given probability vector $(p_i)_{i=1}^M$ refers to the unique Borel probability measure satisfying $\nu=\sum_{i=1}^Mp_i\nu\circ f_i^{-1}$. We say that $(f_i)_{i=1}^M$ satisfies the SSC if $f_i(E),1\leq i\leq M$, are pairwise disjoint.
As general probability measures do not have particular geometric structure like self-similar sets, and one can hardly assume any separation condition for their support, it is even very difficult to examine for what probability measures the weaker version (\ref{weaker}) of Gersho's conjecture holds.

In the present paper, we will prove (\ref{weaker}) for Ahlfors-David probability measures on $\mathbb{R}^q$. Recall that a Borel probability measure $\mu$ on $\mathbb{R}^q$ is said to be $s_0$-dimensional Ahlfors-David regular if there exist constants $C_1,C_2,\epsilon_0>0$, such that
\begin{equation}\label{dameasure}
 C_1\epsilon^{s_0}\leq\mu(B(x,\epsilon))\leq C_2\epsilon^{s_0},\;{\rm for\;all}\;\;\epsilon\in(0,\epsilon_0),\;x\in{\rm supp}(\mu),
\end{equation}
where $B(x,\epsilon):=\{y\in\mathbb{R}^q: d(y,x)\leq\epsilon\}$. We denote by $K$ the support of $\mu$. Then by (\ref{dameasure}), $K$ is clearly compact. When $s_0\notin\mathbb{N}$, the measure $\mu$ is singular with respect to Lebesgue measure. By \cite[Lemma 12.3]{GL:00}, for $C':=2^{s_0}\max\{C_2,\epsilon_0^{-s_0}\}$, we have
\begin{equation}\label{s5}
\sup_{x\in\mathbb{R}^q}B(x,\epsilon)\leq C'\epsilon^{s_0}\;\mbox{for\;all}\;\;\epsilon\in(0,\infty).
\end{equation}
For simplicity, we will assume that (\ref{s5}) holds with $C_2$ in place of $C'$.

Next, we recall some known results by Graf and Luschgy regarding the quantization for Ahlfors-David measures. For this, we need some more definitions.

For a Borel probability measure $\nu$, the $s$-dimensional upper quantization coefficient $\overline{Q}_r^{s}(\nu)$ for $\nu$ of order $r$ and the lower one $\underline{Q}_r^{s}(\nu)$ are defined by
\[
\overline{Q}_r^{s}(\nu):=\limsup_{n\to\infty}n^{\frac{r}{s}}e^r_{n,r}(\nu),\;
\underline{Q}_r^{s}(\nu):=\liminf_{n\to\infty}n^{\frac{r}{s}}e^r_{n,r}(\nu),\;s\in(0,\infty).
\]
The upper (lower) quantization dimension $\overline{D}_r(\nu)$ ($\underline{D}_r(\nu)$) is exactly the critical point at which the upper (lower) quantization coefficient jumps from infinity to zero. According to \cite[Proposition 11.3]{GL:00} and \cite{PK:01}, we have
\[
\overline{D}_r(\nu)=\limsup_{n\to\infty}\frac{\log n}{-\log e_{n,r}(\nu)},\;\underline{D}_r(\nu)=\liminf_{n\to\infty}\frac{\log n}{-\log e_{n,r}(\nu)}.
\]
The upper and lower quantization coefficient and the upper and lower quantization dimension are natural characterizations for the asymptotic properties of the quantization error. We refer to \cite{GL:00,GL:02,GL:05,KZ:16,Kr:08,LM:02,MiRk:15,PK:01,Zhu:16} for some related results in this direction.

By \cite[Theorem 12.18]{GL:00}, for the measures $\mu$ satisfying (\ref{dameasure}), we have
\begin{equation}\label{gl}
0<\underline{Q}_r^{s_0}(\mu)\leq\overline{Q}_r^{s_0}(\mu)<\infty;\;\;\mbox{implying}\;\;e^r_{n,r}(\mu)\asymp n^{-\frac{r}{s_0}}.
\end{equation}

Next, we state our main result of the paper. For a finite $\alpha\subset\mathbb{R}^q$ and a Voronoi partition $\{P_b(\alpha)\}_{b\in\alpha}$ with respect to $\alpha$, we write (cf. (\ref{integral}))
\begin{eqnarray*}
\underline{J}(\alpha,\mu):=\min_{b\in\alpha}I_b(\alpha,\mu);\;\overline{J}(\alpha,\mu):=\max_{b\in\alpha}I_b(\alpha,\mu).
\end{eqnarray*}
\begin{theorem}\label{mthm}
Let $\mu$ be an $s_0$-dimensional Ahlfors-David probability measure on $\mathbb{R}^q$. For every $n$, let $\alpha_n$ be an arbitrary $n$-optimal set for $\mu$ of order $r$ and $\big(P_a(\alpha_n)\big)$ an arbitrary voronoi partition with respect to $\alpha_n$. Then we have
\begin{eqnarray*}
&\underline{J}(\alpha_n,\mu),\overline{J}(\alpha_n,\mu)\asymp\frac{1}{n}e^r_{n,r}(\mu);\;e^r_{n,r}(\mu)-e^r_{n+1,r}(\mu)\asymp \frac{1}{n}e^r_{n,r}(\mu).
\end{eqnarray*}
\end{theorem}
\begin{remark}
By Theorem \ref{mthm} and (\ref{gl}), we conclude that $\underline{J}(\alpha_n,\mu),\overline{J}(\alpha_n,\mu)$ and the error difference
$e^r_{n,r}(\mu)-e^r_{n+1,r}(\mu)$ are of the same order as $n^{-(1+\frac{r}{s_0})}$.
\end{remark}

As the support $K$ of $\mu$ generally does not have particular geometric structure like self-similar sets, and no separation condition for the support is assumed, we will fix an integer $m\geq 2$ and consider the largest number of pairwise disjoint closed balls of radii $m^{-k}$ which are centered in $K$. The advantage of doing so is that, we may shrink or expand such a closed ball to some suitable size without losing control of the $\mu$-measures. We will make use of some auxiliary measures by pushing forward and pulling back the conditional measures of $\mu$ on suitable neighborhoods of the above-mentioned balls and establish a series of preliminary lemmas regarding the quantization errors. One may see \cite{Zhu:13} for more applications of auxiliary measures of this type. The proof for the main result will rely on the following two aspects.

First, for a given $\alpha_n\in C_{n,r}(\mu)$, we will choose a suitable integer $k$ and establish upper and lower bounds for the number of points of $\alpha_n$ in suitable neighborhoods of those balls of radii $m^{-k}$. These bounds will allow us to show that each element $P_a(\alpha_n)$ of a Voronoi partition $\{P_a(\alpha_n)\}_{a\in\alpha_n}$ intersects at most a bounded number (independent of $n$) of the above-mentioned balls. This enables us to estimate $\overline{J}(\alpha_n,\mu)$ from above.

Secondly, for an arbitrary point $a\in\alpha_n$, we will choose a bounded number of points in $\alpha_n$, and show that the union $U$ of the corresponding elements of $\{P_a(\alpha_n)\}$ contains a neighborhood of one of the above-mentioned closed ball, and $U\cap K$ is contained in a bounded number (independent of $n$) of neighborhoods of such closed balls. This, together with \cite[Theorem 4.1]{GL:00} and our preliminary results will enable us to establish a lower estimate for $\underline{J}(\alpha_n,\mu)$.

\section{Preliminary lemmas}

In the remaining part of the paper, we denote by $\mu$ the Ahlfors-David measure satisfying (\ref{dameasure}) and $K$ the support of $\mu$. Let $m\geq 2$ be a fixed integer. In this section, we will establish some preliminary lemmas, some of which will be given in a more general context and stated in terms of $\nu$. We set
\[
k_0:=\min\{k: 2m^{-k}<\epsilon_0\}.
 \]

For every $k\geq k_0$, we denote by $\phi_k$ the largest number of pairwise disjoint closed balls of radii $m^{-k}$ which are centered in $K$. We choose such $\phi_k$ closed balls and denote them by $E_\sigma,\sigma\in\Omega_k$, where
\begin{equation*}
\Omega_k:=\{(k,1),(k,2),\ldots,(k,\phi_k)\}.
\end{equation*}
For every $\sigma\in\Omega_k$, let $c_\sigma$ denote the center of $E_\sigma$. We write
\begin{equation}\label{cylinder}
A_\sigma:=B(c_\sigma,|E_\sigma|)=B(c_\sigma,2 m^{-k}),\;\;D_\sigma:=B\big(c_\sigma,\frac{7}{16}|E_\sigma|\big)=B\big(c_\sigma,\frac{7}{8} m^{-k}\big);
\end{equation}
where $|A|$ denotes the diameter of a set $A\subset\mathbb{R}^q$. We have the following simple fact:
\begin{lemma}\label{first}
There exists a constant $N$ such that $\phi_k\leq\phi_{k+1}\leq N\phi_k$ for all $k\geq k_0$.
\end{lemma}
\begin{proof}
By the definition, we have $K\subset\bigcup_{\sigma\in\Omega_k}A_\sigma$. Thus, by (\ref{dameasure}), we deduce
\begin{eqnarray*}
\phi_k C_1m^{-ks_0}\leq 1\leq \phi_k C_22^{s_0}m^{-ks_0}.
\end{eqnarray*}
It follows that
\begin{equation}\label{phikestimate}
C_2^{-1}2^{-s_0}m^{ks_0}\leq \phi_k\leq C_1^{-1}m^{ks_0}.
\end{equation}
Hence, we have $\phi_k\leq\phi_{k+1}\leq C_1^{-1}C_22^{s_0}m\phi_k$.
It suffices to set $N:=C_1^{-1}C_22^{s_0}m$.
\end{proof}

As a consequence of (\ref{gl}) and (\ref{phikestimate}), for $n\asymp\phi_k$, we have
\begin{eqnarray}\label{onenth}
\frac{1}{n}e^r_{n,r}(\mu)\asymp n^{-(1+\frac{r}{s_0})}\asymp\phi_k^{-(1+\frac{r}{s_0})}\asymp \big(m^{ks_0}\big)^{-(1+\frac{r}{s_0})}=m^{-k(s_0+r)}.
\end{eqnarray}

The subsequent three lemmas are given in a more general context.
\begin{lemma}\label{pre1}
Let $\eta>0$. There exists an integer $M(\eta)$, such that for every Borel probability measure $\nu$ on $\mathbb{R}^q$ with compact support $K_\nu$, we have
\[
\sup_{n\geq M(\eta)-1}e^r_{n,r}(\nu)\leq(\eta|K_\nu|)^r.
\]
\end{lemma}
\begin{proof}
Let $N_\eta(K_\nu)$ denote the largest number of pairwise disjoint closed balls of radii $\frac{1}{2}|K_\nu|\eta$, which are centered in $K_\nu$. Then we double the radii of these balls and get a cover for $K_\nu$. By estimating the volumes, we have
\[
N_\eta(K_\nu)(2^{-1}\eta|K_\nu|)^q\leq((1+2\eta)|K_\nu|)^q.
\]
This implies that $N_\eta(K_\nu)\leq(2\eta^{-1}+4)^q$. Let $[x]:=\max\{k\in\mathbb{Z}:k\leq x\}$. We set
\[
M(\eta):=[(2\eta^{-1}+4)^q]+1.
\]
Then $K_\nu$ can be covered by $M(\eta)-1$ closed balls of radius $\eta|K_\nu|$. We denote by $\beta$ the centers of such $M(\eta)-1$ closed balls. It follows that
\[
\sup_{n\geq M(\eta)-1}e^r_{n,r}(\nu)\leq e^r_{ M(\eta)-1,r}(\nu)\leq\int d(x,\beta)^rd\nu(x)\leq (\eta|K_\nu|)^r.
\]
This completes the proof of the lemma.
\end{proof}

With the following lemma, we give an upper estimate for the error difference $e_{k-1,r}^r(\nu)-e_{k,r}^r(\nu)$, provided that $\nu$ satisfies a certain local property.

\begin{lemma}\label{pre2}
Let $\nu$ be a Borel probability measure on $\mathbb{R}^q$ with compact support $K_\nu$ such that
$\sup_{x\in \mathbb{R}^q}\nu(B(x,\epsilon))\leq C\epsilon^{t}$ for every $\epsilon\in(0,\infty)$. Assume that $|K_\nu|\leq1$.
Then for each $k\geq 2$, there exists a real number $\zeta_{k,r}$, which depends on $C, t$ and $k$, such that
\[
e_{k-1,r}^r(\nu)-e_{k,r}^r(\nu)\geq \zeta_{k,r}.
\]
\end{lemma}
\begin{proof}
Let $\alpha_{k-1}=\{a_i\}_{i=1}^{k-1}\in C_{k-1,r}(\nu)$. Set
\[
\delta_{k,1}:=(4(k-1)C)^{-\frac{1}{t}},\;\delta_{k,2}:=(2(k-1)C)^{-\frac{1}{t}}.
\]
Then for every $1\leq i\leq k-1$, we have
\begin{eqnarray}\label{s4}
\nu(B(a_i,\delta_{k,1}))\leq C\delta_{k,1}^t=\frac{1}{4(k-1)};\;\nu(B(a,\delta_{k,2}))\leq C\delta_{k,2}^t=\frac{1}{2(k-1)}.
\end{eqnarray}
By estimating the volumes, we may find an integer $l_k$ which depends on $k$ and $C$, such that $K_\nu\setminus \bigcup_{i=1}^{k-1}B(a_i,\delta_{k,2})$ can be covered by $l_k$ closed balls $B_i,1\leq i\leq l_k$, of radii
$\delta:=2^{-1}\min\{\delta_{k,2}-\delta_{k,1},\delta_{k,1}\}$
which are centered in $K_\nu\setminus \bigcup_{i=1}^{k-1}B(a_i,\delta_{k,2})$. For each $1\leq i\leq l_k$, we denote by $b_i$ the center of $B_i$. By (\ref{s4}), we have
\[
\nu\bigg(K_\nu\setminus \bigcup_{i=1}^{k-1}B(a_i,\delta_{k,2})\bigg)\geq\frac{1}{2}.
\]
Hence, there exists some $B_i$ with $\nu(B_i)\geq(2l_k)^{-1}$. Set $\beta_k:=\alpha_{k-1}\cup\{b_i\}$.
Then
\begin{eqnarray}\label{pre2a}
e_{k-1,r}^r(\nu)-e_{k,r}^r(\nu)&\geq&I(\alpha_{k-1},\nu)-I(\beta,\nu)\nonumber\\&\geq&\int_{B_i}d(x,\alpha_{k-1})^rd\nu(x)-\int_{B_i}d(x,b_i)^rd\nu(x).
\end{eqnarray}
Note that $B_i$ does not intersect any one of the balls $B(a_i,\delta_{k,1}),1\leq i\leq k-1$. Thus,
\[
\inf_{x\in B_i}d(x,\alpha_{k-1})\geq \delta_{k,1},\;\;\sup_{x\in B_i}d(x,b_i)\leq \frac{1}{2}\delta_{k,1}.
\]
This, together with (\ref{pre2a}), yields
\begin{eqnarray*}
e_{k-1,r}^r(\nu)-e_{k,r}^r(\nu)\geq\nu(B_i)(\delta_{k,1}^r-2^{-r}\delta_{k,1}^r)\geq \frac{1}{2l_k}(1-2^{-r})\delta_{k,1}^r.
\end{eqnarray*}
The lemma follows by setting $\zeta_{k,r}:=\frac{1}{2l_k}(1-2^{-r})\delta_{k,1}^r$.
\end{proof}

Our next lemma is based on some results in \cite{GL:00} and \cite{GL:04}. This lemma will be used in the proof of the main result.
\begin{lemma}\label{microapp}
Let $\nu$ satisfy the assumption in Lemma \ref{pre2}. Then for every $n\geq 1$, there exists a number $d_n>0$ which depends on $n$ and $C$, such that
 \begin{eqnarray*}
 \inf_{\alpha_n\in C_{n,r}(\nu)}\underline{J}(\alpha_n,\nu)>d_n.
 \end{eqnarray*}
 \end{lemma}
 \begin{proof}
 Let $\alpha_n\in C_{n,r}(\nu)$. We write $\{P_a(\alpha_n)\}_{a\in\alpha_n}$ for
 an arbitrary Voronoi partition with respect to $\alpha_n$. By \cite[Theorem 4.11]{GL:00}, we have
 \[
 {\rm card}(\alpha_n)=n,\;\; {\rm and}\;\;\min_{a\in\alpha_n}\nu(P_a(\alpha_n))>0.
  \]
  Set $\epsilon_1:=2^{-1}C^{-1/t}$ and $\{a\}=\alpha_1\in C_{1,r}(\nu)$. We have
 \begin{eqnarray*}
 \int d(x,a)^rd\nu(x)\geq\int_{\mathbb{R}^q\setminus B(a,\epsilon_1)}d(x,a)^rd\nu(x)\geq (1-2^{-t})\epsilon_1^r=:d_n(1).
 \end{eqnarray*}

 Let $n\geq 2$ and $a\in\alpha_n$.  One can easily see that
 \[
\sup_{x\in K_\nu}d(x,\alpha_n)\leq2|K_\nu|.
\]
we set $\beta:=\alpha_n\setminus\{a\}$. Choose an arbitrary $b\in\beta$ and $y\in P_b(\alpha_n)$. Then for every $x\in P_a(\alpha_n)\cap K_\nu$,
\[
d(x,\beta)\leq d(x,y)+d(y,\alpha_n)\leq3|K_\nu|\leq 3.
\]
\begin{eqnarray*}
 e^r_{n-1,r}(\nu)&\leq&\int d(x,\beta)^rd\nu(x)\\&=&\sum_{b\in\beta}\int_{P_b(\alpha_n)}d(x,\beta)^rd\nu(x)+\int_{P_a(\alpha_n)}d(x,\beta)^rd\nu(x)\\
 &\leq&\sum_{b\in\beta}\int_{P_b(\alpha_n)}d(x,b)^rd\nu(x)+\nu(P_a(\alpha_n))3^r\\
 &=&\sum_{b\in\beta}\int_{P_b(\alpha_n)}d(x,\alpha_n)^rd\nu(x)+3^r\nu(P_a(\alpha_n)).
\end{eqnarray*}
  Using this and Lemma \ref{pre2}, we deduce
\begin{eqnarray*}
 \zeta_{n,r}\leq e^r_{n-1,r}(\nu)-e^r_{n,r}(\nu)\leq\nu(P_a(\alpha_n))3^r-\int_{P_a(\alpha_n)}d(x,\beta)^rd\nu(x).
\end{eqnarray*}
It follows that $\nu(P_a(\alpha_n))\geq 3^{-r} \zeta_{n,r}$. Following \cite[Proposition 12.12]{GL:00} and define
 \begin{equation*}
 \delta_B:=\inf\{\epsilon:\;\nu(B^\circ(a,\delta))\geq\frac{1}{2}\nu(P_a(\alpha_n))\},
 \end{equation*}
 where $B^\circ(a,\delta):=\{x\in\mathbb{R}^q:d(x,a)<\delta\}$. We have
 \[
 \sup_{x\in\mathbb{R}^q}\nu(B^\circ(x,\epsilon))\leq\sup_{x\in\mathbb{R}^q}\nu(B(x,\epsilon))\leq C\epsilon^t.
 \]
 Then by (12.13) and (12.14) of \cite{GL:00}, we have
 \[
 \nu(B^\circ(x,\delta_B))\leq\frac{1}{2}\nu(P_a(\alpha_n))\;\;\mbox{and}\;\;\delta_B\geq\big(\frac{1}{2}C^{-1}\nu(P_a(\alpha_n))\big)^{\frac{1}{t}}.
 \]
 Using this, we further deduce
 \begin{eqnarray*}
 I_a(\alpha_n,\nu)&\geq&\int_{P_a(\alpha_n)\setminus B^\circ(a,\delta_B)}d(x,a)^rd\nu(x)
 \\&\geq&\frac{1}{2}\nu(P_a(\alpha_n))\delta_B^r\\&\geq&\frac{1}{2}(C^{-1}2^{-1})^{\frac{r}{t}}\nu(P_a(\alpha_n)))^{1+\frac{r}{t}}
 \\&\geq&2^{-(1+\frac{r}{t})}C^{-\frac{r}{t}}(3^{-r} \zeta_{n,r})^{(1+\frac{r}{t})}=:d_n(2).
 \end{eqnarray*}
 The lemma follow by setting $d_n:=\min\{d_n(1),d_n(2)\}$.
 \end{proof}

 Now we return to the Ahlfors-David measure $\mu$ satisfying (\ref{dameasure}). For $\delta>0$, let $(F)_\delta$ denote the closed $\delta$-neighborhood of a set $F\subset\mathbb{R}^q$. Let $E_\sigma,A_\sigma,D_\sigma$ be as defined in (\ref{cylinder}) and $B$ a Borel set satisfying $D_\sigma\subset B\subset(A_\sigma)_{\frac{19}{16}|A_\sigma|}$. Then
\begin{eqnarray}\label{z5}
\frac{7}{16}|A_\sigma|\leq|B|\leq|(A_\sigma)_{\frac{19}{16}|A_\sigma|}|\leq\frac{27}{8}|A_\sigma|.
\end{eqnarray}
Also, by (\ref{dameasure}) and (\ref{z5}), we have
\begin{eqnarray}
\mu(B)\left\{ \begin{array}{ll}
\leq\mu\big((A_\sigma)_{\frac{19}{16}|A_\sigma|}\big)\leq C_2\big(\frac{27}{16}|A_\sigma|\big)^{s_0}\leq C_2\big(\frac{27}{7}\big)^{s_0}|B|^{s_0}\\\\
\geq\mu(D_\sigma)\geq C_1\big(\frac{7}{32}|A_\sigma|\big)^{s_0}\geq C_1\big(\frac{7}{108}\big)^{s_0}|B|^{s_0}
\end{array}\right..
\end{eqnarray}
For the largest choice $(A_\sigma)_{\frac{19}{16}|A_\sigma|}$ and the smallest one $D_\sigma$, we have
\begin{eqnarray}\label{z6}
\mu(D_\sigma)|D_\sigma|^r&\leq&\mu\big((A_\sigma)_{\frac{19}{16}|A_\sigma|}\big)|(A_\sigma)_{\frac{19}{16}|A_\sigma|}|^r
\leq\bigg(\frac{54C_2}{7C_1}\bigg)^{s_0+r}\mu(D_\sigma)|D_\sigma|^r.
\end{eqnarray}
Let $h_B$ be an arbitrary similitude of similarity ratio $|B|$. We define
\begin{equation}\label{z7}
\lambda_B:=\mu(\cdot|B)\circ h_B,\;{\rm implying}\;\;\mu(\cdot|B)=\lambda_B\circ h_B^{-1}.
\end{equation}
We denote by $K_B$ the support of $\lambda_B$. Then we have $|K_B|\leq1$.
\begin{lemma}\label{pre3}
Let $\mu$ be the Ahlfors-David measure satisfying (\ref{dameasure}). Assume that $B$ is a Borel set with $\xi|B|^{s_0}\leq\mu(B)\leq \xi^{-1}|B|^{s_0}$. Then there exists a constant $\xi_B>0$, such that for every $\epsilon>0$, we have
$\sup_{x\in K_B}\lambda_B(B(x,\epsilon))\leq \xi_B\epsilon^{s_0}$.
\end{lemma}
\begin{proof}
For every $x\in K_B$ and $\epsilon>0$, we have
\begin{eqnarray*}
\lambda_B(B(x,\epsilon))&=&\frac{1}{\mu(B)}\mu(h_B(B(x,\epsilon)\cap B))\\
&=&\frac{1}{\mu(B)}\mu((B(h_B(x),\epsilon|B|)\cap B))\\
&\leq&\frac{1}{\xi|B|^{s_0}}C_2(\epsilon|B|)^{s_0}\\
&=&\xi^{-1}C_2\epsilon^{s_0}.
\end{eqnarray*}
The lemma follows by setting $\xi_B:=\xi^{-1}C_2$.
\end{proof}

In the following, we will need to consider measures $\lambda_B$ for different sets $B$. For convenience, we write
\[
\lambda_B=:\left\{ \begin{array}{ll}
\nu_{\sigma,1}\;\;&\mbox{if}\;\;B=E_\sigma\\
\nu_{\sigma,2}\;\;&\mbox{if}\;\;B=(E_\sigma)_{\frac{1}{16}|E_\sigma|}\\
\nu_{\sigma,3}\;\;&\mbox{if}\;\;B=(A_\sigma)_{\frac{1}{8}|A_\sigma|}\\
\nu_{\sigma,4}\;\;&\mbox{if}\;\;B=D_\sigma
\end{array}\right.;\;h_B=:\left\{ \begin{array}{ll}
h_{\sigma,1}\;\;&\mbox{if}\;\;B=E_\sigma\\
h_{\sigma,2}\;\;&\mbox{if}\;\;B=(E_\sigma)_{\frac{1}{16}|E_\sigma|}\\
h_{\sigma,3}\;\;&\mbox{if}\;\;B=(A_\sigma)_{\frac{1}{8}|E_\sigma|}\\ h_{\sigma,4}\;\;&\mbox{if}\;\;B=D_\sigma
\end{array}\right..
\]
\begin{remark}
By Lemma \ref{pre1}, we are able to define a first constant $n_1$ which will be useful later. Let $M(\eta)$ be as defined in the proof of Lemma \ref{pre1}. We define
\begin{equation}
\eta_0:=C_1^{\frac{1}{r}}C_2^{-\frac{1}{r}}(18)^{-(1+\frac{s_0}{r})}\;\;\mbox{and}\;\;n_1:=M(\eta_0).
\end{equation}
Then for any probability $\nu$ with $|K_\nu|\leq1$, by Lemma \ref{pre1}, we have
\begin{equation}\label{s6}
e^r_{m_1-1,r}(\nu)\leq\eta_0^r=C_1C_2^{-1}(18)^{-(r+s_0)}.
\end{equation}
\end{remark}

When $(A_\sigma)_{\frac{1}{16}|A_\sigma|}\cap (A_\omega)_{\frac{1}{8}|A_\omega|}\neq\emptyset$ for some distinct $\sigma,\omega\in\Omega_k$, we will need to consider the following two choices of $B$:
\[
(A_\omega)_{\frac{1}{8}|A_\omega|}\setminus(E_\sigma)_{\frac{1}{16}|E_\sigma|},\;
(A_\omega)_{\frac{1}{8}|A_\omega|}\setminus E_\sigma.
\]
For these two choices of $B$, we write
\begin{eqnarray*}
\lambda_B=:\left\{ \begin{array}{ll}
\nu_{\omega,5}\;\;&\mbox{if}\;\;B=(A_\omega)_{\frac{1}{8}|A_\omega|}\setminus(E_\sigma)_{\frac{1}{16}|E_\sigma|}\\
\nu_{\omega,6}\;\;&\mbox{if}\;\;B=(A_\omega)_{\frac{1}{8}|A_\omega|}\setminus E_\sigma.
\end{array}\right.;\\h_B=:\left\{ \begin{array}{ll}
h_{\omega,5}\;\;&\mbox{if}\;\;B=(A_\omega)_{\frac{1}{8}|A_\omega|}\setminus(E_\sigma)_{\frac{1}{16}|E_\sigma|}\\
h_{\omega,6}\;\;&\mbox{if}\;\;B=(A_\omega)_{\frac{1}{8}|A_\omega|}\setminus E_\sigma.
\end{array}\right..
\end{eqnarray*}
Also, for the proof of the main result, we will consider a larger neighborhood $(A_\omega)_{\frac{19}{16}|A_\omega|}$ of $A_\omega$. So for $B=(A_\omega)_{\frac{19}{16}|A_\omega|}$, we also write
\[
\lambda_B=:\nu_{\omega,7},\;\;h_B:=h_{\omega,7}.
\]

With the next lemma, we define two more constants $n_2,n_3$. For this, we set
\[
n_0:=[(130)^q],\;\;k_1:=[(82)^q],\;\;k_2:=[(218)^q];\;k_3:=\big[\big(\frac{35}{4}\big)^q\big].
\]
By estimating the volumes, we know that for every $\sigma\in\Omega_k$, the set $(A)_\sigma$ can be covered by $n_0$ closed balls of radii $\frac{1}{64}|A_\sigma|$ which are centered in $A_\sigma$. In fact, let $n_0$ be the largest number of pairwise disjoint closed balls of radii $\frac{1}{128}|A_\sigma|$ which are centered in $A_\sigma$. Then we have
\[
n_0\bigg(\frac{1}{128}|A_\sigma|\bigg)^q\leq \bigg((1+\frac{1}{64})|A_\sigma|\bigg)^q.
\]
Similarly, one can see that the set $(A_\sigma)_{\frac{1}{8}|A_\sigma|}$ can be covered by $k_1$ closed balls of radii $\frac{1}{32}|A_\sigma|$ which are centered in $(A_\sigma)_{\frac{1}{8}|A_\sigma|}$; $(A_\sigma)_{\frac{19}{16}|A_\sigma|}$ can be covered by $k_2$ closed balls of radii $\frac{1}{32}|A_\sigma|$ which are centered in $(A_\sigma)_{\frac{19}{16}|A_\sigma|}$.
\begin{lemma}\label{pre5}
For $\sigma\in\Omega_k$, let $\nu_{\sigma,i},1\leq i\leq 7$, be defined as above. Then
\begin{enumerate}
\item[(1)] For $k\geq 1$, there exists a $\zeta_{k,r}>0$ which is independent of $\sigma$, such that
\[
\min_{1\leq i\leq 4}\big(e_{k-1,r}^r(\nu_{\sigma,i})-e_{k,r}^r(\nu_{\sigma,i})\big)\geq \zeta_{k,r}.
\]
\item[(2)] There exists an integer $n_2>n_1+k_1$, such that for $\omega\in\Omega_k$, and $i=3,5,6$,
\[
\sup_{n\geq n_2-n_1-k_1}e^r_{n,r}(\nu_{\omega,i})<\bigg(\frac{7C_1}{54C_2}\bigg)^{s_0+r}\zeta_{n_1,r}.
\]
\item[(3)]
There exists an integer $n_3>N(n_0+n_2)k_3+k_2=:n_4$, such that
\[
\sup_{n\geq n_3-n_4}e^r_{n,r}(\nu_{\omega,7})<\bigg(\frac{7C_1}{54C_2}\bigg)^{s_0+r}\zeta_{N(n_0+n_2),r},\;\;\mbox{for\;every} \;\omega\in\Omega_k.
\]
\end{enumerate}
\end{lemma}
\begin{proof}
(1) Let $k\geq k_0$ and $\sigma\in\Omega_k$. Then by (\ref{dameasure}), we have
\begin{eqnarray*}
C_1\bigg(\frac{7}{40}\bigg)^{s_0}|(A_\sigma)_{\frac{1}{8}|E_\sigma|}|^{s_0}&\leq&\mu(D_\sigma)\leq\mu(E_\sigma)
\leq\mu\big((E_\sigma)_{\frac{1}{16}|E_\sigma|}\big)
\\&\leq&\mu\big((A_\sigma)_{\frac{1}{8}|A_\sigma|}\big)\leq C_2\bigg(\frac{10}{7}\bigg)^{s_0}|D|^{s_0}.
\end{eqnarray*}
Set $\xi:=\frac{40}{7}$. Then for this $\xi$, Lemma \ref{pre3} holds for all the following choices for $B$:
\[
D_\sigma, E_\sigma, (E_\sigma)_{\frac{1}{16}|E_\sigma|},(A_\sigma)_{\frac{1}{8}|E_\sigma|}.
\]
Thus, according to Lemma \ref{pre2}, (1) follows by setting
\[
\zeta_{k,r}:=\frac{1}{2l_k}(1-2^{-r})(4(k-1)\xi)^{-\frac{1}{s_0}}.
\]

(2) Note that $|K_{\nu_{\omega,3}}|\leq1$. By Lemma \ref{pre1}, (2) follows by setting
\[
\eta:=\bigg(\frac{7C_1}{54C_2}\bigg)^{\frac{s_0+r}{r}}\zeta_{n_1,r}^{\frac{1}{r}},\;\;n_2:=M(\eta)+k_1+n_1+1.
\]

(3) By Lemma \ref{pre1}, it suffices to set
\[
\eta:=\bigg(\frac{7C_1}{54C_2}\bigg)^{\frac{s_0+r}{r}}\zeta_{N(n_0+n_2),r}^{\frac{1}{r}},\;\;n_3:=M(\eta)+n_4+1.
\]
This completes the proof of the lemma.
\end{proof}

\section{A characterization for the $n$-optimal sets for $\mu$}
 For each $n\geq (n_0+n_2)\phi_1$, there exists a unique $k$ such that
 \[
 (n_0+n_2)\phi_k\leq n<(n_0+n_2)\phi_{k+1}.
 \]

 Next, we fix an arbitrary $\alpha_n\in C_{n,r}$. We need to establish a characterization for the positions where the points of $\alpha_n$ are lying.  Write
 \[
 \alpha_n(1):=\alpha_n\setminus\bigcup_{\sigma\Omega_k}(A_\sigma)_{\frac{1}{16}|A_\sigma|},\;L_c:=\mbox{card}(\alpha_n(1)),
 \;\;\alpha_n(2):=\alpha_n\setminus\alpha_n(1).
 \]
\begin{lemma}\label{ess1}
We have $L_c\leq n_0\phi_k$.
\end{lemma}
\begin{proof}
Suppose that $L_c>n_0\phi_k$. we deduce a contradiction. Write
\[
F_n:=\{x\in K: d(x,\alpha_n)=d(x,\alpha_n(1))\}.
\]
We distinguish two cases.

Case 1: $F_n=\emptyset$. In this case, we set $\beta:=\alpha_n(2)$. Then we have ${\rm card}(\beta)<n$ and $I(\beta,\mu)=I(\alpha_n,\mu)$. This contradicts the optimality of $\alpha_n$.

Case 2: $F_n\neq\emptyset$. Then for each $\sigma\in\Omega_k$, we denote by $\gamma_\sigma$ the set of the centers of $n_0$ closed balls of radii $\frac{|A_\sigma|}{64}$ which is centered in $A_\sigma$ and cover $A_\sigma$. We set
\[
\beta:=\alpha_n(2)\cup\bigg(\bigcup_{\sigma\in\Omega_k}\gamma_\sigma\bigg).
\]
Then we have ${\rm card}(\beta)<n$. We have $K\subset \bigcup_{\sigma\in\Omega_k}A_\sigma$ and
\begin{equation}\label{g01}
\sup_{x\in A_\sigma}d(x,\beta)\leq \frac{|A_\sigma|}{64}<\frac{|A_\sigma|}{16}.
\end{equation}
Note that $\beta\supset\alpha_n(2)$. So by (\ref{g01}), we have $d(x,\beta)\leq d(x,\alpha_n)$ for all $x\in K$. For every $x\in F_n\cap A_\sigma$ and $y\in B(x,\frac{|A_\sigma|}{32})$, we have $d(x,\alpha_n)\geq\frac{1}{16}|A_\sigma|$ and
\[
d(y,\alpha_n)\geq d(x,\alpha_n)-d(x,y)\geq \frac{|A_\sigma|}{16}-\frac{|A_\sigma|}{32}=\frac{|A_\sigma|}{32}.
\]
We fix an arbitrary $x_0\in F_n\cap A_\sigma$. We have
\begin{eqnarray*}
I(\alpha_n,\mu)-I(\beta,\mu)&\geq&\int_{B(x_0,\frac{|A_\sigma|}{32})}d(x,\alpha_n)^r-d(x,\beta)^rd\mu(x)\\
&\geq&\big((32)^{-r}-(64)^{-r}\big)|A_\sigma|^r\mu(B(x_0,\frac{|A_\sigma|}{32})
>0.
\end{eqnarray*}
It follows that $I(\alpha_n,\mu)>I(\beta,\mu)$, contradicting the optimality of $\alpha_n$.
\end{proof}

For $\sigma\in\Omega_k$ and $\beta\subset\mathbb{R}^q$, we write
\[
G_\sigma:=\big\{x\in E_\sigma\cap K: d(x,\beta)=d(x,\beta\setminus (A_\sigma)_{\frac{1}{16}|A_\sigma|)}\big\}.
\]

\begin{lemma}\label{pre6}
Let $\sigma\in\Omega_k$ and $\beta\subset\mathbb{R}^q$ with ${\rm card}(\beta\cap (A_\sigma)_{\frac{1}{16}|A_\sigma|})<n_1$. We have
\[
\left\{ \begin{array}{ll}
\int_{(E_\sigma)_{\frac{|E_\sigma|}{16}}}d(x,\beta)^rd\mu(x)\geq\mu\big((E_\sigma)_{\frac{|E_\sigma|}{16}}\big)|(E_\sigma)_{\frac{|E_\sigma|}{16}}|^r
e^r_{n_1-1,r}(\nu_{\sigma,2})&\mbox{if}\;\;G_\sigma\neq\emptyset\\
\int_{E_\sigma}d(x,\beta)^rd\mu(x)\geq\mu((E_\sigma)|E_\sigma|^r
e^r_{n_1-1,r}(\nu_{\sigma,1})&\mbox{if}\;\;G_\sigma=\emptyset
\end{array}\right..
\]
\end{lemma}
\begin{proof}
Assume that $G_\sigma\neq\emptyset$. Then there exists some $x_0\in E_\sigma\cap K$, such that
\[
d(x_0,\beta)>\big(\frac{1}{4}+\frac{1}{16}\big)|A_\sigma|=\frac{5}{8}|E_\sigma|.
\]
Thus, for every $x\in B(x_0,\frac{1}{16}|E_\sigma|)$, we have
\[
d(x,\beta)\geq d(x_0,\beta)-d(x,x_0)>\big(\frac{5}{8}-\frac{1}{16}\big)|E_\sigma|=\frac{9}{16}|E_\sigma|=\frac{1}{2}\big|(E_\sigma)_{\frac{|E_\sigma|}{16}}\big|.
\]
Hence, for $x\in h_{\sigma,2}^{-1}(B(x_0,\frac{1}{16}|E_\sigma|)$, we have
\begin{equation}\label{g02}
d(x,h_{\sigma,2}^{-1}(\beta)\geq\frac{1}{2}.
\end{equation}
Note that $B(x_0,\frac{1}{16}|E_\sigma|)\subset (E_\sigma)_{\frac{|E_\sigma|}{16}}$ and $|E_\sigma|<\epsilon_0$. We have
\begin{eqnarray}\label{g03}
\nu_{\sigma,2}\big(h_{\sigma,2}^{-1}(B(x_0,\frac{1}{16}|E_\sigma|))\big)&=&\frac{\mu\big(B(x_0,\frac{1}{16}|E_\sigma|)
\cap(E_\sigma)_{\frac{|E_\sigma|}{16}}\big)}{\mu((E_\sigma)_{\frac{|E_\sigma|}{16}})}=\frac{\mu\big(B(x_0,\frac{1}{16}|E_\sigma|)
\big)}{\mu((E_\sigma)_{\frac{|E_\sigma|}{16}})}\nonumber\\&\geq& \frac{C_1\big(\frac{1}{16}|E_\sigma|\big)^{s_0}}{C_2\big(\frac{9}{16}|E_\sigma|\big)^{s_0}}
=C_1C_2^{-1}\frac{1}{9^{s_0}}.
\end{eqnarray}
Using (\ref{s6}), (\ref{g02}) and (\ref{g03}) we deduce
\begin{eqnarray*}
\int_{(E_\sigma)_{\frac{|E_\sigma|}{16}}}d(x,\beta)^rd\mu(x)&=&\mu\big((E_\sigma)_{\frac{|E_\sigma|}{16}}\big)|(E_\sigma)_{\frac{|E_\sigma|}{16}}|^r
\int d(x,h_{\sigma,2}^{-1}(\beta))^rd\nu_{\sigma,2}(x)\\&\geq&\mu\big((E_\sigma)_{\frac{|E_\sigma|}{16}}\big)|(E_\sigma)_{\frac{|E_\sigma|}{16}}|^r
\nu_{\sigma,2}\big(h_{\sigma,2}^{-1}(B(x_0,\frac{1}{16}|E_\sigma|))\big)\big(\frac{1}{2}\big)^r\\
&>&\mu\big((E_\sigma)_{\frac{|E_\sigma|}{16}}\big)|(E_\sigma)_{\frac{|E_\sigma|}{16}}|^rC_1C_2^{-1}\frac{1}{9^{s_0}}\frac{1}{2^r}\\
&>&\mu\big((E_\sigma)_{\frac{|E_\sigma|}{16}}\big)|(E_\sigma)_{\frac{|E_\sigma|}{16}}|^re^r_{n_1-1,r}(\nu_{\sigma,2}).
\end{eqnarray*}

Next, we assume that $G_\sigma=\emptyset$. Then for every $x\in E_\sigma\cap K$, we have
\begin{equation}\label{temp2}
d(x,\beta)=d(x,\beta\cap (A_\sigma)_{\frac{1}{16}|A_\sigma|}).
\end{equation}
By the hypothesis, ${\rm card}(\beta\cap (A_\sigma)_{\frac{1}{16}|A_\sigma|})<n_1$. Using this and (\ref{temp2}), we deduce
\begin{eqnarray*}
\int_{E_\sigma}d(x,\beta)^rd\mu(x)&=&\mu(E_\sigma)|E_\sigma|^r
\int d(x,h_{\sigma,2}^{-1}(\beta))^rd\nu_{\sigma,1}(x)\\&\geq&\mu(E_\sigma)|E_\sigma|^re^r_{n_1-1,r}(\nu_{\sigma,1}).
\end{eqnarray*}
This completes the proof of the lemma.
\end{proof}

For each $\sigma\in\Omega_k$, we write $L_\sigma:={\rm card}\big(\alpha_n\cap(A_\sigma)_{\frac{1}{16}|A_\sigma|}\big)$. We have
\begin{lemma}\label{key1}
For every $\sigma\in\Omega_k$, we have $L_\sigma\geq n_1$.
\end{lemma}
\begin{proof}
Suppose that $L_\sigma<n_1$ for some $\sigma\in\Omega_k$. We will deduce a contradiction. Note that $n\geq\phi_k(n_0+n_2)$ and $n_2>n_1$.
By Lemma \ref{ess1}, we have
\[
{\rm card}(\alpha_n(2))-L_\sigma\geq(n_0+n_2)\phi_k-n_0\phi_k-n_1\geq(\phi_k-1)n_2.
\]
Hence, there exists some $\omega\in\Omega_k$ with $L_\omega\geq n_2$. Next, we distinguish two cases.

\emph{Case 1}: $(A_\omega)_{\frac{1}{8}|A_\omega|}\cap (A_\sigma)_{\frac{1}{16}|A_\sigma|}=\emptyset$.
Let $\beta_\omega$ denote the set of the centers of $k_1$ closed balls of radii $\frac{1}{32}|A_\omega|$ which are centered in $(A_\omega)_{\frac{1}{8}|A_\omega|}$ and cover $(A_\omega)_{\frac{1}{8}|A_\omega|}$. We have two subcases.

\emph{Case (1a)}: $G_\sigma\neq\emptyset$. In this case, we set
\begin{eqnarray*}
&&\gamma_{L_\omega-k_1-n_1}(\nu_{\omega,3})\in C_{L_\omega-k_1-n_1,r}(\nu_{\omega,3}),\;\beta_{n_1}(\nu_{\sigma,2})\in C_{n_1,r}(\nu_{\sigma,2}).
\\&&\beta:=\big(\alpha_n\setminus(A_\omega)_{\frac{1}{16}|A_\omega|}\big)\cup \beta_\omega\cup h_{\omega,3}(\gamma_{L_\omega-k_1-n_1}(\nu_{\omega,3}))\cup h_{\sigma,2}(\beta_{n_1}(\nu_{\sigma,2})).
\end{eqnarray*}
Then we have ${\rm card}(\beta)\leq n$. By triangle inequality, one can see that
\[
d(x,\beta)\leq d(x,\alpha_n),\;\mbox{for\;every}\;\;x\in K\setminus(A_\omega)_{\frac{1}{8}|A_\omega|}.
\]
As a consequence, for $J:=K\setminus(A_\omega)_{\frac{1}{8}|A_\omega|}$, we have
\begin{eqnarray}\label{temp1}
\int_J d(x,\beta)^rd\mu(x)\leq\int_J d(x,\alpha_n)^rd\mu(x).
\end{eqnarray}
This allows us to focus on $(A_\omega)_{\frac{1}{8}|A_\omega|}$ and  $(E_\sigma)_{\frac{1}{16}|E_\sigma|}$.
By the supposition, we have $L_\sigma<n_1$. Thus, by Lemma \ref{pre6}, we deduce
\[
\int_{(E_\sigma)_{\frac{1}{16}|E_\sigma|}}d(x,\alpha_n)^rd\mu(x)\geq \mu\big((E_\sigma)_{\frac{1}{16}|E_\sigma|}\big)|(E_\sigma)_{\frac{1}{16}|E_\sigma|}|^re^r_{n_1-1,r}(\nu_{\sigma,2}).
\]
Note that, for $x\in (E_\sigma)_{\frac{1}{16}|E_\sigma|}$, we have $d(x,\beta)\leq d(x,h_{\sigma,2}(\beta_{n_1}(\nu_{\sigma,2})))$. Hence,
\begin{eqnarray}\label{s0}
\Delta_1(\alpha_n,\beta):&=&\int_{(E_\sigma)_{\frac{1}{16}|E_\sigma|}}d(x,\alpha_n)^r-d(x,\beta)^rd\mu(x)\nonumber\\
&\geq&\mu\big((E_\sigma)_{\frac{1}{16}|E_\sigma|}\big)|(E_\sigma)_{\frac{1}{16}|E_\sigma|}|^r(e^r_{n_1-1,r}(\nu_{\sigma,2})-e^r_{n_1,r}(\nu_{\sigma,2})).
\end{eqnarray}
On the other hand, we have
\begin{eqnarray}\label{s1}
\Delta_2(\alpha_n,\beta):&=&\int_{(A_\omega)_{\frac{1}{8}|A_\omega|}}d(x,\beta)^r-d(x,\alpha_n)^rd\mu(x)\nonumber
\\&\leq&\int_{(A_\omega)_{\frac{1}{8}|A_\omega|}}d(x,\beta)^rd\mu(x)\nonumber\\
&\leq&\mu((A_\omega)_{\frac{1}{8}|A_\omega|})|(A_\omega)_{\frac{1}{8}|A_\omega|}|^re^r_{n_2-k_1-n_1,r}(\nu_{\omega,3}).
\end{eqnarray}
By (\ref{z6}) and Lemma \ref{pre5}, we have $\Delta_2(\alpha_n,\beta)<\Delta_1(\alpha_n,\beta)$. This, together with (\ref{temp1}), implies
$I(\beta,\mu)<I(\alpha_n,\mu)$, contradicting the optimality of $\alpha_n$.

\emph{Case (1b)}: $G_\sigma=\emptyset$. In this case, let $\gamma_{L_\omega-k_1-n_1}(\nu_{\omega,3})\in C_{L_\omega-k_1-n_1,r}(\nu_{\omega,3})$ and define
\[
\beta:=\big(\alpha_n\setminus(A_\omega)_{\frac{1}{16}|A_\omega|}\big)\cup \beta_\omega\cup h_{\omega,3}(\gamma_{L_\omega-k_1-n_1}(\nu_{\omega,3}))\cup h_{\sigma,1}(\beta_{n_1}(\nu_{\sigma,1})).
\]
Then (\ref{temp1}) and (\ref{s1}) remain true. This allows us to focus on $(A_\omega)_{\frac{1}{8}|A_\omega|}$ and $E_\sigma$. Since $G_\sigma=\emptyset$ and $L_\sigma<n_1$, by Lemma \ref{pre6}, we have
\begin{eqnarray}\label{s2}
\Delta_3(\alpha_n,\beta):&=&\int_{E_\sigma}d(x,\alpha_n)^r-d(x,\beta)^rd\mu(x)\nonumber\\
&\geq&\mu(E_\sigma)|E_\sigma|^r(e^r_{n_1-1,r}(\nu_{\sigma,1})-e^r_{n_1,r}(\nu_{\sigma,1})).
\end{eqnarray}
By (\ref{z6}) and Lemma \ref{pre5}, we have $\Delta_2(\alpha_n,\beta)<\Delta_3(\alpha_n,\beta)$. This together with (\ref{temp1}) implies that $I(\alpha_n,\mu)>I(\beta,\mu)$, which contradicts the optimality of $\alpha_n$.

\emph{Case 2:} $(A_\omega)_{\frac{1}{8}|A_\omega|}\cap (A_\sigma)_{\frac{1}{16}|A_\sigma|}\neq\emptyset$. Let $\beta_\omega$ be the same as in Case (i).
Due to the set $\beta_\omega$, (\ref{temp1}) remains true.

\emph{Case (2a)}: $G_\sigma\neq\emptyset$. In this case, (\ref{s0}) remains true.
 We set
\begin{eqnarray*}
&&\gamma_{L_\omega-k_1-n_1}(\nu_{\omega,5})\in C_{L_\omega-k_1-n_1,r}(\nu_{\omega,5}),\;\beta_{n_1}(\nu_{\sigma,2})\in C_{n_1,r}(\nu_{\sigma,2});\\
&&\beta:=\big(\alpha_n\setminus(A_\omega)_{\frac{1}{16}|A_\omega|}\big)\cup \beta_\omega\cup h_{\omega,5}(\gamma_{L_\omega-k_1-n_1}(\nu_{\omega,5}))\cup h_{\sigma,2}(\beta_{n_1}(\nu_{\sigma,2})).
\end{eqnarray*}
For integrals over $(A_\omega)_{\frac{1}{8}|A_\omega|}\setminus(E_\sigma)_{\frac{1}{16}|E_\sigma|}$, we have
\begin{eqnarray*}
\Delta_4(\alpha_n,\beta):&=&\int_{(A_\omega)_{\frac{1}{8}|A_\omega|}\setminus(E_\sigma)_{\frac{1}{16}|E_\sigma|}}d(x,\beta)^r-d(x,\alpha_n)^rd\mu(x)
\nonumber
\\&\leq&\int_{(A_\omega)_{\frac{1}{8}|A_\omega|}\setminus(E_\sigma)_{\frac{1}{16}|E_\sigma|}}d(x,\beta)^rd\mu(x)\nonumber\\
&\leq&\mu\big((A_\omega)_{\frac{1}{8}|A_\omega|}\setminus(E_\sigma)_{\frac{1}{16}|E_\sigma|}\big)\big|(A_\omega)_{\frac{1}{8}|A_\omega|}
\setminus(E_\sigma)_{\frac{1}{16}|E_\sigma|}\big|^re^r_{n_2-n_1-k_1,r}(\nu_{\omega,5}).
\end{eqnarray*}
By (\ref{z6}), (\ref{s0}) and Lemma \ref{pre5}, we have $\Delta_1(\alpha_n,\beta)>\Delta_4(\alpha_n,\beta)$. Using this and (\ref{temp1}), we deduce that $I(\beta,\mu)<I(\alpha_n,\mu)$. This contradicts the optimality of $\alpha_n$.

\emph{Case (2b)}: $G_\sigma=\emptyset$. In this case, (\ref{temp1}) and (\ref{s2}) remain true. We set
\[
\beta:=\big(\alpha_n\setminus(A_\omega)_{\frac{1}{16}|A_\omega|}\big)\cup \beta_\omega\cup h_{\omega,6}(\gamma_{L_\omega-k_1-2n_1}(\nu_{\omega,6}))\cup h_{\sigma,1}(\beta_{n_1}(\nu_{\sigma,1})).
\]
Due to (\ref{temp1}), we focus on $(A_\omega)_{\frac{1}{8}|A_\omega|}\setminus E_\sigma$ and $E_\sigma$. We have
\begin{eqnarray*}
\Delta_5(\alpha_n,\beta):&=&\int_{(A_\omega)_{\frac{1}{8}|A_\omega|}\setminus E_\sigma}d(x,\beta)^r-d(x,\alpha_n)^rd\mu(x)\nonumber
\\&\leq&\int_{(A_\omega)_{\frac{1}{8}|A_\omega|}\setminus E_\sigma}d(x,\beta)^rd\mu(x)\nonumber\\
&\leq&\mu\big((A_\omega)_{\frac{1}{8}|A_\omega|}\setminus E_\sigma\big)\big|(A_\omega)_{\frac{1}{8}|A_\omega|}\setminus E_\sigma\big|^re^r_{n_2-n_1-k_1,r}(\nu_{\omega,6}).
\end{eqnarray*}
Using this, (\ref{z6}), (\ref{s2}) and Lemma \ref{pre5}, we deduce that $\Delta_5(\alpha_n,\beta)<\Delta_3(\alpha_n,\beta)$. Thus, by (\ref{temp1}), we have $I(\beta,\mu)<I(\alpha_n,\mu)$, contradicting the optimality of $\alpha_n$.
\end{proof}
\begin{corollary}\label{cor1}
Let $\sigma\in\Omega_k$ and $a\in\alpha_n$. If $P_a(\alpha_n)\cap A_\sigma\neq\emptyset$, then
\begin{eqnarray*}
d(a,c_\sigma)\leq\frac{13}{8}|A_\sigma|.
\end{eqnarray*}
\end{corollary}
\begin{proof}
By the hypothesis, there exists an $x\in A_\sigma$ with $d(x,\alpha_n)=d(x,a)$. By Lemma \ref{key1}, we have $\alpha_n\cap(A_\sigma)_{\frac{1}{16}|A_\sigma|}\neq\emptyset$, implyng that $d(x,\alpha_n)\leq\frac{9}{8}|A_\sigma|$. Hence,
\begin{eqnarray*}
\frac{9}{8}|A_\sigma|\geq d(x,\alpha_n)=d(x,a)\geq d(a,c_\sigma)-d(x,c_\sigma)\geq d(a,c_\sigma)-\frac{1}{2}|A_\sigma|.
\end{eqnarray*}
It follows that $d(a,c_\sigma)\leq\frac{13}{8}|A_\sigma|$.
\end{proof}

For the proof of our main theorem, we need to establish an upper bound for the number of points $a$ in $\alpha_n$ such that $P_a(\alpha_n)\cap A_\sigma\cap K\neq\emptyset$. Write
\begin{equation}\label{g06}
\mathcal{M}_\sigma:=\{a\in\alpha_n:P_a(\alpha_n)\cap A_\sigma\cap K\neq\emptyset\};\;M_\sigma:={\rm card}(\mathcal{M}_\sigma).
\end{equation}
Then by Corollary \ref{cor1}, we have $\mathcal{M}_\sigma\subset (A_\sigma)_{\frac{9}{8}|A_\sigma|}$. It follows that
\begin{equation}\label{sgzhu1}
{\rm card}(\alpha_n\cap(A_\sigma)_{\frac{9}{8}|A_\sigma|})\geq M_\sigma.
\end{equation}
\begin{remark}
Note that, for distinct words $\sigma,\tau\in\Omega_k$, the balls $A_\sigma,A_\tau$ may be overlapping. Thus, little can be said about $L_\sigma$ even if $M_\omega$ is "excesively large". Fortunately,  $E_\sigma,\sigma\in\Omega_k$, are pairwise disjoint. Hence, if $M_\omega$ is "too large" for some $\omega\in\Omega_k$, then as we will see, ${\rm card}(\alpha_n\cap E_\sigma)$ would be "too small". We will use this fact to give an upper bound for $M_\sigma$.
\end{remark}
\begin{lemma}\label{key2}
For every $\sigma\in\Omega_k$, we have $M_\sigma\leq n_3$.
\end{lemma}
\begin{proof}
Suppose that $M_\omega>n_3$ for some $\omega\in\Omega_k$. We deduce a contradiction. Set
 \[
 \mathcal{N}_\omega:=\{\tau\in\Omega_k: E_\tau\cap (A_\omega)_{(\frac{9}{8}+\frac{1}{16})|A_\omega|}\neq\emptyset\};\;\;N_\omega:={\rm card}(\mathcal{N}_\omega).
 \]
By estimating volumes, we have $N_\omega\leq [(\frac{35}{4})^q]=k_3$. Since $n_3>n_4$, we deduce
\begin{eqnarray*}
n-M_\omega&<&(n_0+n_2)\phi_{k+1}-n_3\\&\leq&N(n_0+n_2)\phi_k-k_3N(n_0+n_2)\\&\leq& N(n_0+n_2)(\phi_k-N_\omega).
\end{eqnarray*}
Since $E_\tau,\tau\in\Omega_k$, are pairwise disjoint, there exists some $\sigma\in\Omega_k$ such that
\[
E_\sigma\cap(A_\omega)_{\frac{19}{16}|A_\omega|}=\emptyset,\;\;{\rm card}(\alpha_n\cap E_\sigma)<N(n_0+n_2).
\]
Let $\widetilde{\beta}_\omega$ be the centers of $k_2$ closed balls of radius $\frac{1}{32}|A_\omega|$ which are centered in $(A_\omega)_{(\frac{9}{8}+\frac{1}{16})|A_\omega|}$ and cover $(A_\omega)_{(\frac{9}{8}+\frac{1}{16})|A_\omega|}$. We need to distinguish two cases. Set
\[
H_\sigma:=\{x\in D_\sigma\cap K: d(x,\alpha_n)=d(x,\alpha_n\setminus E_\sigma)\}.
\]

\emph{Case 1}: $H_\sigma\neq\emptyset$. In this case, we set
\begin{eqnarray*}
&\gamma_{M_\omega-n_4}(\nu_{\omega,7})\in C_{M_\omega-n_4,r}(\nu_{\omega,7}),\;\beta_{N(n_0+n_2)}(\nu_{\sigma,1})\in C_{N(n_0+n_2),r}(\nu_{\sigma,1});\\
&\beta:=\big(\alpha_n\setminus(A_\omega)_{\frac{9}{8}|A_\omega|}\big)\cup \widetilde{\beta}_\omega\cup h_{\omega,7}(\gamma_{M_\omega-n_4}(\nu_{\omega,7}))\cup h_{\sigma,1}(\beta_{N(n_0+n_2)}(\nu_{\sigma,1})).
\end{eqnarray*}
Then by (\ref{sgzhu1}), we have ${\rm card}(\beta)\leq n$. Due to the set $\widetilde{\beta}_\omega$, we have
\begin{equation}\label{g04}
d(x,\beta)\leq d(x,\alpha_n)\;\;{\rm for\;every}\;\;x\in K\setminus(A_\omega)_{\frac{19}{16}|A_\omega|}.
\end{equation}
Hence, we may focus on $(A_\omega)_{\frac{19}{16}|A_\omega|}$ and $E_\sigma$. Since $H_\sigma\neq\emptyset$, we choose an $x_0\in D_\sigma\cap K$ such that $d(x_0,\alpha_n)>\frac{1}{16}|E_\sigma|$. Note that $B(x_0,\frac{1}{16}|E_\sigma|)\subset E_\sigma$. Since $|K_B|\leq 1$ for $B=E_\sigma$, by (\ref{s6}), we have
\begin{eqnarray*}
\int_{E_\sigma}d(x,\alpha_n)^rd\mu(x)&=&\mu(E_\sigma)|E_\sigma|^r\int d(x,h^{-1}_{\sigma,1}(\alpha_n))^rd\nu_{\sigma,1}(x)\\
&\geq&\mu(E_\sigma)|E_\sigma|^r\int_{h^{-1}_{\sigma,1}(B(x_0,\frac{1}{16}|E_\sigma|))} d(x,h^{-1}_{\sigma,1}(\alpha_n))^rd\nu_{\sigma,1}(x)\\
&\geq&\mu(E_\sigma)|E_\sigma|^rC_1C_2^{-1}\frac{1}{(16)^{s_0}}\frac{1}{(16)^r}
\\&>&\mu(E_\sigma)|E_\sigma|^re^r_{n_1-1,r}(\nu_{\sigma,1}).
\end{eqnarray*}
Note that $d(x,\beta)\leq d(x, h_{\sigma,1}(\beta_{N(n_0+n_2)}(\nu_{\sigma,1})))$ for $x\in E_\sigma$. We further deduce
\begin{eqnarray*}
\Delta_6(\beta,\alpha_n):&=&\int_{E_\sigma}d(x,\alpha_n)^r-d(x,\beta)^rd\mu(x)\\
&\geq&\mu(E_\sigma)|E_\sigma|^r(e^r_{n_1-1,r}(\nu_{\sigma,1})-e^r_{N(n_0+n_2),r}(\nu_{\sigma,1}))\\
&>&\mu(E_\sigma)|E_\sigma|^r(e^r_{N(n_0+n_2)-1,r}(\nu_{\sigma,1})-e^r_{N(n_0+n_2),r}(\nu_{\sigma,1})).
\end{eqnarray*}
For the integrals over $(A_\omega)_{\frac{19}{16}|A_\omega|}$, we have
\begin{eqnarray}\label{s7}
\Delta_7(\beta,\alpha_n):&=&\int_{(A_\omega)_{\frac{19}{16}|A_\omega|}}d(x,\beta)^r-d(x,\alpha_n)^rd\mu(x)\leq
\int_{(A_\omega)_{\frac{19}{16}|A_\omega|}}d(x,\beta)^rd\mu(x)\nonumber\\
&\leq&\int_{(A_\omega)_{\frac{19}{16}|A_\omega|}}d(x,h_{\omega,7}(\gamma_{M_\omega-n_4}(\nu_{\omega,7})))^rd\mu(x)\nonumber\\
&=&\mu\big((A_\omega)_{\frac{19}{16}|A_\omega|}\big)\big|(A_\omega)_{\frac{19}{16}|A_\omega|}\big|^re^r_{M_\omega-n_4,r}(\nu_{\omega,7}).
\end{eqnarray}
Thus, by (\ref{z6}) and Lemma \ref{pre5}, we have $\Delta_6(\alpha_n,\beta)>\Delta_7(\alpha_n,\beta)$. This, together with (\ref{g04}), yields that $I(\beta,\mu)<I(\alpha_n,\mu)$, which contradicts the optimality of $\alpha_n$.

\emph{Case 2}: $H_\sigma=\emptyset$. Let $\widetilde{\beta}_\omega$ be the same as in Case 1. We set
\begin{eqnarray*}
&\gamma_{M_\omega-n_4}(\nu_{\omega,7})\in C_{M_\omega-n_4,r}(\nu_{\omega,7}),\;\beta_{N(n_0+n_2)}(\nu_{\sigma,4})\in C_{N(n_0+n_2),r}(\nu_{\sigma,4});\\
&\beta:=\big(\alpha_n\setminus(A_\omega)_{\frac{1}{16}|A_\omega|}\big)\cup \widetilde{\beta}_\omega\cup h_{\omega,7}(\gamma_{M_\omega-n_4}(\nu_{\omega,7}))\cup h_{\sigma,4}(\beta_{N(n_0+n_2)}(\nu_{\sigma,4})).
\end{eqnarray*}
Due to the set $\widetilde{\beta}_\omega$, (\ref{g04}) remains true. We focus on $(A_\omega)_{\frac{19}{16}|A_\omega|}$ and $D_\sigma$. We have
\begin{eqnarray*}
\Delta_8(\beta,\alpha_n):&=&\int_{D_\sigma}d(x,\alpha_n)^r-d(x,\beta)^rd\mu(x)\\
&\geq&\mu(D_\sigma)|D_\sigma|^r(e^r_{N(n_0+n_2)-1,r}(\nu_{\sigma,4})-e^r_{N(n_0+n_2),r}(\nu_{\sigma,4})).
\end{eqnarray*}
 Note that (\ref{s7}) remains true. Hence, by (\ref{z6}) and Lemma \ref{pre5}, we deduce that $\Delta_8(\alpha_n,\beta)>\Delta_7(\alpha_n,\beta)$. So, by (\ref{g04}), it follows that $I(\beta,\mu)<I(\alpha_n,\mu)$, contradicting the optimality of $\alpha_n$. The proof of the lemma is complete.
\end{proof}
\section{Proof of Theorem \ref{mthm}}
Let $(n_0+n_2)\phi_k\leq n<(n_0+n_2)\phi_{k+1}$. Next, we give an upper estimate for $\overline{J}(\alpha_n,\mu)$.
\begin{lemma}\label{pf1}
There exists a constant $C_3>0$ such that for every $a\in\alpha_n$, we have
\[
I_a(\alpha_n,\mu)\leq C_3 m^{-k(s_0+r)}.
\]
\end{lemma}
\begin{proof}
Let $a$ be an arbitrary point in $\alpha_n$. As above, we denote by $c_\sigma$ the center of $A_\sigma$ for $\sigma\in\Omega_k$. By Corollary \ref{cor1}, for every $\sigma\in\Omega_k$ with $A_\sigma\cap P_a(\alpha_n)\neq \emptyset$, we have $d(a,c_\sigma)\leq \frac{13}{8}|A_\sigma|$.
We write
\[
\Gamma_k(a):=\{\tau\in\Omega_k: d(a,c_\tau)\leq\frac{13}{8}|A_\tau|\};\;N_a:={\rm card}(\Gamma_k(a)).
\]
Then we have $P_a(\alpha_n)\cap K\subset\bigcup_{\tau\in\Gamma_k(a)}A_\sigma$. Note that $E_\tau$ and $A_\tau$ share the same center, and $E_\tau,\tau\in\Omega_k$, are pairwise disjoint. By estimating Volumes,
\[
\big(\frac{13}{4}+\frac{1}{2})|A_\sigma|\big)^q\geq N_a\big(\frac{1}{2}|A_\sigma|\big)^q.
\]
It follows that $N_a\leq\big(\frac{15}{2}\big)^q$. On the other hand, by Lemma \ref{key1}, we know that $\alpha_n\cap (A_\tau)_{\frac{1}{16}|A_\tau|}\neq\emptyset$ for all $\tau\in\Omega_k$. Thus, for very $x\in P_a(\alpha_n)\cap A_\tau$, we have $d(x,a)\leq d(x,\alpha_n)\leq\frac{9}{8}|A_\tau|$. We further deduce
\begin{eqnarray*}
I_a(\alpha_n,\mu)&=&\int_{P_a(\alpha_n)}d(x,a)^rd\mu(x)\\
&\leq&\sum_{\tau\in\Gamma_k(a)}\int_{P_a(\alpha_n)\cap A_\tau}d(x,a)^rd\mu(x)\\&\leq&\sum_{\tau\in\Gamma_k(a)}\mu(A_\tau\cap P_a(\alpha_n))\big(\frac{9}{8}\big)^r |A_\tau|^r\\
&\leq&\big(\frac{9}{8}\big)^r \sum_{\tau\in\Gamma_k(a)}\mu(A_\tau)|A_\tau|^r\\&\leq&\big(\frac{9}{8}\big)^r \big(\frac{15}{2}\big)^q C_2m^{-k(s_0+r)}2^r.
\end{eqnarray*}
The lemma follows by setting $C_3:=C_22^{s_0+r}\big(\frac{9}{8}\big)^r\big(\frac{15}{2}\big)^q$.
\end{proof}

Next, we give a lower estimate for $\underline{J}(\alpha_n,\mu)$.
\begin{lemma}\label{pf2}
There exists a constant $C_6>0$ such that for every $a\in\alpha_n$, we have
\[
I_a(\alpha_n,\mu)\geq C_6 m^{-k(s_0+r)}.
\]
\end{lemma}
\begin{proof}
Let $a$ be an arbitrary point in $\alpha_n$. We write
\[
S_a:=\{\sigma\in\Omega_k:A_\sigma\cap P_a(\alpha_n)\cap K\neq\emptyset\}.
\]
By the proof of Lemma \ref{pf1}, we have ${\rm card}(S_a)\leq N_a\leq\big(\frac{15}{2}\big)^q$. Let
$\mathcal{M}_\sigma$ be as defined in (\ref{g06}).
By Lemma \ref{key2}, we have ${\rm card}(\mathcal{M}_\sigma)\leq n_3$. Now we choose an arbitrary $\sigma^0\in S_a$ and define
\begin{eqnarray*}
\beta(a):=\{a\}\cup\bigg(\bigcup_{\sigma\in S_a}\mathcal{M}_\sigma\bigg),\;\;G_a:=\{x\in K: d(x,\beta(a))=d(x,\alpha_n)\};\;\;\widetilde{G}_a:=A_{\sigma^0}\cup(G_a\cap K).
\end{eqnarray*}
Then we have $t_a:={\rm card}(\beta(a))\leq n_3\big(\frac{15}{2}\big)^q$ and
\begin{equation}\label{gatilde}
G_a=\bigcup_{b\in\beta(a)}P_b(\alpha_n),\;\;\mu(\cdot|G_a)=\mu(\cdot|\widetilde{G}_a).
\end{equation}
Also, by Corollary \ref{cor1}, we have
\begin{equation}\label{gadiameter}
|A_{\sigma^0}|\leq|\widetilde{G}_a|\leq6\cdot\frac{13}{8}\cdot |A_\sigma|=\frac{39}{4}|A_{\sigma^0}|.
\end{equation}
Note that $G_a$ contains the closed ball $A_{\sigma^0}$ and $\widetilde{G}_a$ is contained in the union of $t_a$ such balls. Hence, there exists a constant $C_4>0$ such that
\[
C_4|\widetilde{G}_a|^{s_0}\leq\mu(\widetilde{G}_a)\leq C_4^{-1}|\widetilde{G}_a|^{s_0}.
\]
Let $\lambda_{\widetilde{G}_a}$ be defined as we did for $\lambda_B$ in (\ref{z7}) and $K_{\widetilde{G}_a}$ its support. Using the same argument as in the proof of Lemma \ref{pre3}, one can find a constant $C_5>0$ such that
\begin{equation}\label{ga}
\sup_{x\in \mathbb{R}^q}\lambda_{\widetilde{G}_a}(B(x,\epsilon))\leq C_5\epsilon^{s_0}.
\end{equation}
According to Theorem 4.1 of \cite{GL:00}, we know that
\[
\beta(a)\in C_{t_a,r}\big(\mu(\cdot|G_a)\big)= C_{t_a,r}\big(\mu(\cdot|\widetilde{G}_a)\big).
\]
Set $\underline{d}:=\min\{d_h:1\leq h\leq n_3\big(\frac{15}{2}\big)^q\}$. Using Lemma \ref{microapp}, we deduce
\begin{eqnarray*}
I_a(\alpha_n,\mu)\geq d_{t_a}\mu(\widetilde{G}_a)|\widetilde{G}_a|^r\geq d_{t_a}C_1m^{-k(s_0+r)}\geq \underline{d} C_1m^{-k(s_0+r)}.
\end{eqnarray*}
The proof of the lemma is complete by setting $C_6:=\underline{d}C_1$.
\end{proof}

\begin{remark}
The measure $\lambda_{\widetilde{G}_a}$ is an amplification of the conditional measure $\mu(\cdot|G_a)$ (cf. (\ref{z7}) and (\ref{gatilde})). The set $A_\sigma$ in the definition of $\widetilde{G}_a=A_\sigma\cup (G_a\cap K)$ is used to guarantee the size of the amplification, so that we can obtain the lower bound for $I_a(\alpha_n,\mu)$ in terms of $m^{-k(s_0+r)}$.
\end{remark}
For two number sequences $(a_n)_{n=1}^\infty$ and $(b_n)_{n=1}^\infty$, we write
$a_n\lesssim b_n$ ($a_n\gtrsim b_n$) if there exists some constant $C$ such that $a_n\leq Cb_n$ ($a_n\geq Cb_n$) for all $n\geq 1$. With the above preparations, we are now able to prove our main result.

\emph{Proof of Theorem \ref{mthm}}

Let $a$ be an arbitrary point of $\alpha_n$. By Lemmas \ref{pf1} and \ref{pf2}, we have
\[
nC_6C_3^{-1} I_a(\mu,\alpha_n)\leq e^r_{n,r}(\mu)=\sum_{b\in\alpha_n}I_b(\alpha_n,\mu)\leq nC_3C_6^{-1} I_a(\mu,\alpha_n).
\]
This implies that, for every $a\in\alpha_n$, we have $I_a(\alpha_n,\mu)\asymp\frac{1}{n}e^r_{n,r}(\mu)$. It follows that
\[
\underline{J}(\alpha_n,\mu),\;\overline{J}(\alpha_n,\mu)\asymp \frac{1}{n}e^r_{n,r}(\mu).
\]

Next, we show the remaining part of the theorem. Let $\alpha_{n+1}\in C_{n+1,r}(\mu)$ with
\begin{equation}\label{g05}
(n_0+n_2)\phi_k\leq n+1<(n_0+n_2)\phi_{k+1}.
 \end{equation}
For $a\in\alpha_{n+1}$, let $\widetilde{G}_a$ and $t_a$ be the same as in the proof of Lemma \ref{pf2}. There exists some $A_\sigma$ such that $P_a(\alpha_{n+1})\cap A_\sigma\cap K\neq\emptyset$. Due to (\ref{g05}), there exists some $\sigma\in\Omega_k$ with $M_\sigma>1$ (cf. (\ref{g06})). So, we may choose $a\in\alpha_{n+1}$ such that $t_a>1$. Let $\lambda_{\widetilde{G}_a},h_{\widetilde{G}_a}$ be defined in the same manner as we did for $\lambda_B,h_B$. Set
\[
\gamma_{t_a-1}\in C_{t_a-1,r}(\lambda_{\widetilde{G}_a}),\;\beta_n:=(\alpha_{n+1}\setminus \beta(a))\cup {h_{\widetilde{G}_a}(\gamma_{t_a-1})}.
\]
Then we have ${\rm card}(\beta_{n+1})\leq n$. Hence,
$e^r_{n,r}(\mu)\leq I(\beta,\mu)$.
For $x\in K\setminus \widetilde{G}_a$, we have $x\notin\bigcup_{b\in \beta(a)}P_b(\alpha_{n+1})$. Since $\alpha_{n+1}\setminus \beta(a)\subset\beta_n$, we have
$d(x,\beta)\leq d(x,\alpha_{n+1})$. Thus,
\begin{equation}\label{z2}
\int_{K\setminus \widetilde{G}_a}d(x,\beta)^rd\mu(x)\leq\int_{K\setminus \widetilde{G}_a}d(x,\alpha_{n+1})^rd\mu(x).
\end{equation}
Write $\Delta_{n,1}(\mu):=e^r_{n,r}(\mu)-e^r_{n+1,r}(\mu)$. Note that $|K_{\widetilde{G}_a}|\leq1$. Using (\ref{z2}), we deduce
\begin{eqnarray*}
\Delta_{n,1}(\mu)&\leq& I(\beta,\mu)-I(\alpha_{n+1},\mu)\leq\int_{\widetilde{G}_a}d(x,\beta)^r-d(x,\alpha_{n+1})^rd\mu(x)\\
&\leq&\int_{\widetilde{G}_a}d(x,\beta)^rd\mu(x)
=\mu(\widetilde{G}_a)|\widetilde{G}_a|^re^r_{t_a-1,r}(\lambda_{\widetilde{G}_a})
\\&\leq&\mu(\widetilde{G}_a)|\widetilde{G}_a|^r.
\end{eqnarray*}
This, together with (\ref{onenth}) and (\ref{gadiameter}), yields
\begin{equation}\label{z3}
\Delta_{n,1}(\mu)\leq n_3\big(\frac{15}{2}\big)^{2q}C_2\big(\frac{39}{4}\big)^{s_0+r}2^{s_0+r}m^{-k(s_0+r)}\lesssim\frac{1}{n}e^r_{n,r}(\mu).
\end{equation}

Now let $\alpha_n\in C_{n,r}(\mu)$ and $a\in\alpha_n$ and $\gamma_{t_a+1}\in C_{t_a+1,r}(\lambda_{\widetilde{G}_a})$. Set
\[
\beta_{n+1}:=(\alpha_n\setminus\beta(a))\cup h_{\widetilde{G}_a}(\gamma_{t_a+1}).
\]
Then ${\rm card}(\beta_{n+1})\leq n+1$. Hence, we have $e^r_{n+1,r}(\mu)\leq I(\beta_{n+1},\mu)$.
Note that (\ref{z2}) remains true. Thus, by (\ref{ga}) and Lemma \ref{pre2} with $k=t_a$, we deduce
\begin{eqnarray*}
\Delta_{n,1}(\mu)&\geq& I(\alpha_n,\mu)-I(\beta_{n+1},\mu)\\
&\geq&\int_{\widetilde{G}_a}d(x,\alpha_n)^r-d(x,\beta_{n+1})^rd\mu(x)\\
&\geq&\int_{\widetilde{G}_a}d(x,\beta_{n+1})^rd\mu(x)-\int_{\widetilde{G}_a}d(x,h_{\widetilde{G}_a}(\gamma_{t_a+1}))^rd\mu(x)
\\&=&\mu(\widetilde{G}_a)|\widetilde{G}_a|^r\big(e^r_{t_a,r}(\lambda_{\widetilde{G}_a})-e^r_{t_a+1,r}(\lambda_{\widetilde{G}_a})\big)
\\&\gtrsim&\mu(\widetilde{G}_a)|\widetilde{G}_a|^r.
\end{eqnarray*}
Note that $\widetilde{G}_a$ contains some closed ball $A_\sigma$ with $\sigma\in\Omega_k$. It follows by (\ref{onenth}) that
\begin{equation}\label{z4}
\Delta_{n,1}(\mu)\gtrsim C_1m^{-k(s_0+r)}\gtrsim\frac{1}{n}e^r_{n,r}(\mu).
\end{equation}
Combining (\ref{z3}) and (\ref{z4}), we obtain
\[
e^r_{n,r}(\mu)-e^r_{n+1,r}(\mu)=\Delta_{n,1}(\mu)\asymp\frac{1}{n}e^r_{n,r}(\mu).
\]
This completes the proof of the theorem.

\end{document}